\documentclass[a4paper,12pt]{article}

\usepackage{amsmath}
\usepackage{enumerate}

\usepackage{amssymb}
\usepackage{amsthm}

\def \D {\,\mathrm{d}}
\def \E {e}
\def \I {\mathbb{I}}
\def \smartqed {}

\def \nn {\nonumber}

\def \N {\mathbb{N}}

\def \R {\mathbb{R}}
\def \ind{1\!\!1}

\def \prob        {\ensuremath{\mathbf{P}}}
\def \expect      {\ensuremath{\mathbf{E}}}

\def \Leb {\mathrm{Leb}}

\def \birth {\tau}                  %\birth _x is the birth time of x.

\def \birthtimes {\Sigma}           %sigma algebra generated by the
\def \bound {M_\probe}              %bound for \probe

\def \childmax {K}                  %maximum number of children of a vertex

\def \contr {\Lambda}               %the contraction coefficient in the metric

\def \deg    {\textup{deg}}         %degree.

\def \eps {\varepsilon}

\def \event {\Omega}                %for indexed subset of the probability space
          
\def \expprobe {m_\probe}           %expectation of \probe according to the limit measure

\def \fracrelat {R}           %\fracrelat _{yk}=\frac{\relat_{yk}}{\relat_y}

\def \glob {\Theta}		 %the limit variable 
\def \globdensity {\pi}         % the density of \glob

\def \ggpdensity {\rho}       % the joint density of (\glob,\glob') in the transition kernel

\def \isancestor {\prec}    %x\isancestoreq y denotes that x is ancestor of y.

\def \isancestoreq {\preceq}    %x\isancestoreq y denotes that either x is 
\def \kernel {P}                    %transition kernel of Markov process

\def \leaves {\partial\vertices}    %the leaves of the complete tree

\def \malt   {\lambda ^*}        %malthusian parameter.

\def \markovlimitmeasure {\nu}      %limit measure of \glob _{y_n}

\def \measure {\mu}                 %the random measure on \leaves

\def \metric {d}                    %the metric in the set of leaves

\def \probe {\phi}                  %test function for weak convergence

\def \relat {\Delta}              %relative number of vertices

\def \tmpmeasure {\tilde{\mu}}     % temporary notation for a measure in a lemma

\def \tree   {\Upsilon}             %the random tree.

\def \trees {\mathcal{G}}           %set of all finite rooted ordered 
\def \ver    {\zeta}                  %uniformly random vertex. 

\def \vertexprob {p}                %vertexweight normalized

\def \vertexweight {T}              %exp(-malt*birth)

\def \vertices {\mathcal{N}}	    %the vertices.

\def \wait {\sigma}                 %\wait _x is the random time that we have
\def \waitglob {V}                  %the vector (\wait, \glob)

\def \weight {w}                    %the weight function.

\def \Ykernel {\tilde\kernel}       %the transition kernel of the extended Markov process

\newcommand{\abs}[1]{\left|{#1}\right|}
\newcommand{\parent}[1]{p({#1})}
\newcommand{\subtree}[2] {{#1}_{\downarrow {#2}}}
\newcommand{\cylinder}[2] {{#1}({#2})}

\newcommand{\crop}[2] {{#1}_{| {#2}}}

\newtheorem{theorem}{Theorem}[section]

\newtheorem{definition}[theorem]{Definition}
\newtheorem{remark}[theorem]{Remark}
\newtheorem{lemma}[theorem]{Lemma}

\newtheorem{proposition}[theorem]{Proposition}
\newtheorem{corollary}[theorem]{Corollary}

\begin{document}

\title{Entropy and Hausdorff Dimension in Random Growing Trees}
\author{Anna Rudas\\ Institute of Mathematics, Technical University of Budapest \\
\textit{rudasa@math.bme.hu} \\
and \\
Imre P\'eter T\'oth\\ MTA-BME Stochastics Research Group\\
and Department of Mathematics and Statistics, University of Helsinki \\ \textit{mogy@math.bme.hu} }

\maketitle

%%%%%%%%%%%%%%%%%%%%%%%%%%%%%%%%%%%%%%%%%%%%%%%%%%%
%\tableofcontents
%%%%%%%%%%%%%%%%%%%%%%%%%%%%%%%%%%%%%%%%%%%%%%%%%%%

\abstract{
We investigate the limiting behaviour of random tree growth in preferential attachment models. The tree stems from a root, and we add vertices to the system one-by-one at random, according to a rule which depends on the degree distribution of the already existing tree. The so-called \textit{weight function}, in terms of which the rule of attachment is formulated, is such that each vertex in the tree can have at most $\childmax$ children.

We define the concept of a certain random measure $\measure$ on the leaves of the limiting tree, which captures a global property of the tree growth in a natural way. We prove that the Hausdorff and the packing dimension of this limiting measure is equal and constant with probability one. Moreover, the local dimension of $\measure$ equals the Hausdorff dimension at $\measure$-almost every point. We give an explicit formula for the dimension, given the rule of attachment. 
}

%%%%%%%%%%%%%%%%%%%%%%%%%%%%%%%%%%%%%%%%%%%%%%%%%%%

\section{Introduction}
\label{sec:Introduction}

We investigate a family of tree growth models in which the tree stems from a root in the beginning, and vertices are added one at a time, the new vertex always attaching to exactly one already existing vertex. The rule by which the new vertex chooses its parent, is
dependent on the degree distribution apparent in the tree at the time the vertex is born. The models can be either in discrete time, when a vertex is born in every second, or in continuous time, then birth times are random. For the problems we discuss, these two versions are equivalent and can be translated into each other (details in Section~\ref{sec:TheModel}). The classical models and results of the area use the discrete setting. However, for the proofs we give, the continuous-time version is much more natural and convenient, so this is what we will use.

This big family of models includes the Barab\'asi-Albert graph \cite{barabasialbert} for example, in which the linear preferential attachment rule reproduces certain phenomena observed in real-world networks (e.g.\,the power law decay of the degree sequence). This property of the Barab\'asi-Albert graph was proved in a mathematically precise way in \cite{bolriospetus} and, independently, in \cite{mori1}. A wider class of models is considered in \cite{krapivskyPRL, krapivskyPRE}, for rigorous results on different cases of this model, see \cite{oliveira, rtv}.

The results mentioned above focus on the local behaviour of the random tree, namely, they give results concerning the neighbourhood of a uniformly random vertex, which is chosen from the tree after a long time of tree evolution. In this paper we concentrate on global properties of the limiting tree.

It is natural to pose the following question. Let us fix a vertex, say the first vertex in the first generation, just above the root. What is the ``limiting success level" of this vertex, compared to the other vertices in the same generation? What we mean by this is the number of descendants of this vertex, after a long time of tree evolution, compared to the number of descendants of its brothers.

Another formulation of the same question is to fix a vertex, let the tree grow for a long time, then choose a vertex uniformly at random from the big tree, and ask the probability that this random vertex is descendant of the fixed vertex. Clearly, if we look at these limiting probabilities for let us say the first generation, we get a distribution, itself being random, that codes an important information of the evolution of the tree.

If one looks at the system of these limiting (as time evolution of the tree tends to infinity) random distributions on the different generations of the tree, it is tempting to ask something about the limiting measure of this system, when letting the generation level tend to infinity. We will define the above concepts properly, and will denote this overall limiting measure by $\measure$.

Having a random measure in our hand, which describes a global property of the limiting infinite system, it is natural to ask about the Hausdorff (and packing) dimension of this measure, for several reasons. First, these are the primary quantities capturing the scaling behaviour of the system, so they appear in statistical and Statistical Physics discussions. Secondly, these can actually be measured in (finite, but big) real systems, so they can be used to check the validity of models, or to tune parameters.

On the other hand, the dimension of the measure depends on a parameter of the underlying metric, which is arbitrary. To rule out this (trivial) dependence, it is usual to ask about the entropy of the limiting measure, which depends on the growth process only. This is the natural equivalent of the dimension from a dynamical point of view.

We prove the following results. 
\begin{enumerate}
\item The limiting entropies (as time tends to infinity) of the random measures on the different generations converge to a constant with probability one, as we let the generation level to infinity. This constant $h$ is called the entropy of the limiting measure $\measure$.

\item The Hausdorff and the packing dimension of the random limiting measure $\measure$ are 
constant and equal with probability one. The entropy and the dimension satisfy the usual simple relation $dimension=\frac{entropy}{Ljapunov\, exponent}$, see (\ref{eq:dimensionAndEntropy}).
Moreover, the local dimension of $\measure$ equals the Hausdorff dimension at
$\measure$-almost every point.

\item Given the so-called \textit{weight function} $\weight$, which determines the rule of the tree growth, we provide an explicit formula for the entropy, and thus for the Hausdorff dimension, in terms of $\weight$.

\end{enumerate}

The key to these results is a Markov process appearing naturally in the construction of a $\measure$-typical leaf of the tree. After some discussion of the tree structure, the Markov property will be easy to see. Some technical difficulties will arise from the non-compactness of the state space.

Our model is special in the sense that we only allow a finite degree for each vertex, but it is general in the sense that after having fixed the maximum number of children $\childmax$ a vertex may have, the weight function $\weight$, which determines the rule of attachment, can be any positive-valued function on $\{0,1,\dots,\childmax-1\}$.

The paper is structured as follows: The model and the results are presented in Section~\ref{sec:NotationAndDefinitions}. This also includes a brief discussion of related models and related results in Section~\ref{sec:related_models}. Section~\ref{sec:mainline} contains the main line of the argument, and ends with the proof of the first two results. Section~\ref{sec:proofs} is devoted to proofs of lemmas which have been used but not proven in Section~\ref{sec:mainline}. Finally, Section~\ref{sec:ComputationOfEntropy} contains the proof of the last result.

%%%%%%%%%%%%%%%%%%%%%%%%%%%%%%%%%%%%%%%%%%%%%%%%%%%

\section{Notation, Definitions and Results}
\label{sec:NotationAndDefinitions}

We consider rooted ordered trees, which are also called family
trees or rooted planar trees in the literature.

In order to refer to these trees it is convenient to use genealogical
phrasing. The tree is thus regarded as the coding of the evolution of
a population stemming from one individual (the root of the tree),
whose ``children'' form the ``first generation'' (these are the
vertices connected directly to the root). 
In general, the edges of the
tree represent parent-child relations, the parent always being the one
closer to the root. The birth order between brothers is also taken
into account, this is represented by the tree being an ordered tree
(planar tree). 

We only consider the case when every vertex can have at most $\childmax \in\N$ children. We assume $K\ge 2$ to avoid the trivial case when only one child is born per parent. (In that case the tree growth is linear and the tree has no interesting structure.)
We use the index set $\I:=\{1,2,\dots ,\childmax\}$, and also use 
$\I^-:=\{0,1,\dots ,\childmax-1\}$.

The vertices are labelled by the set
\begin{equation*}
\vertices = \bigcup _{n=0}^{\infty} \I ^{n},\quad 
\text{where}\quad
\I^{0}:=\{\emptyset\}\;,
\end{equation*}
as follows.
$\emptyset$ denotes the root of the tree, its first-born child is
labelled by 1, the second one by 2, etc., and its last one by $\childmax$,
all the vertices in the first
generation are thus labelled with the elements of $\I$. Similarly, in
general, the children of $x=(i_1,i_2,\dots,i_n)$ are labelled by
$(i_1,i_2,\dots,i_n,1)$, $(i_1,i_2,\dots,i_n,2)$, etc. Thus, if a
vertex has label $x=(i_1,i_2,\dots,i_n)\in \vertices$, then it is the
$i_n^{\text{th}}$ child of its parent, which is the
$i_{n-1}^{\text{th}}$ child of its own parent and so on.  If
$x=(i_1,i_2,\dots,i_n)$ and $y=(j_1,j_2,\dots,j_l)$ then we will use
the shorthand notation $xy$ for the concatenation
$(i_1,i_2,\dots,i_n,j_1,j_2,\dots,j_l)$, and with a slight abuse of
notation for $i\in \I$, we use $xi$ for $(i_1,i_2,\dots,i_n,i)$.

There is a natural partial ordering $\isancestor$ on $\vertices$, namely, 
$x\isancestor z$ if $x$ is ancestor of $z$, so if $\exists y\in
\vertices$, $y \ne \emptyset$ such that $z=xy$. We use $x\isancestoreq z$
meaning $x\isancestor z$ or $x=z$.

We can identify a rooted ordered tree with the set of labels of the
vertices, since this set already identifies the set of
edges in the tree. It is clear that a subset $G \subset \vertices$ may
represent a rooted ordered tree iff $\emptyset \in G$, and for each
$(i_1,i_2,\dots,i_n)\in G$ we have $(i_1,i_2,\dots,i_n-1)\in G$ if
$i_n>1$, and $(i_1,i_2,\dots,i_{n-1})\in G$ if $i_n=1$.

We also think of $\vertices$ as the complete rooted ordered tree.

$\trees$ will denote the set of all finite, rooted ordered
trees. The \textit{degree} of vertex $x\in G$ will
denote the number of its children in $G$:
\begin{equation*}
\deg (x,G):=\text{max}\{i\in \I:\, xi\in G\} \quad \text{(zero if $x1\notin G$)}
\end{equation*}

The \emph{subtree} rooted at a vertex $x\in G$ is:
\begin{equation*}
\subtree{G}{x}:= \{y: xy\in G\}\;,
\end{equation*}
this is just the progeny of $x$ viewed as a rooted ordered tree.

% sub %%%%%%%%%%%%%%%%%%%%%%%%%%%%%%%%%%%%%%%%%%%%%%%%%%%%%%%%%%%%%%%%%

\subsection{The Model}
\label{sec:TheModel}

\subsubsection{Continuous-time Model}
Given a function $\weight: \I^- \to \R_+$, referred to as the weight function,
our randomly growing tree $\tree(t)$
is a continuous-time, time-homogeneous Markov chain on
the countable state space $\trees$, with initial state
$\tree(0)=\{\emptyset\}$ and right-continuous trajectories.

The jump rates are the following. Suppose that at some $t\ge 0$ we have
$\tree(t-)=G$, then for each $x\in G$ which has $\deg(x,G)=j<\childmax$,
the process may jump to $G\cup \{xi\}$ with rate
$\weight(\deg(x,G))$ where $i=j+1$. This means
that each existing vertex $x\in \tree(t-)$ `gives birth to a child'
with rate $\weight (\deg (x, \tree (t-)))$, independently of the
others, and stops reproducing when reaches $\deg(x, \tree (t))=\childmax$. 

The Markov chain $\tree(t)$ is well defined for $t\in[0,\infty)$, it does not 
blow up in finite time (see comment at (\ref{eq:globexpectation})).

We define the \emph{total weight} of a tree $G\in\trees$ as
\begin{equation*}
W(G):=\sum_{x\in G} \weight(\deg(x,G))\;.
\end{equation*}
Described in other words, the Markov chain $\tree(t)$ evolves as
follows: assuming $\tree(t-)=G$, at time $t$ a new vertex is added to
it with total rate $W(G)$, and it is attached with an edge to exactly
one already existing vertex, which is $x\in G$ with probability
\begin{equation*}
\frac{\weight(\deg (x,G))} {\sum_{y\in G}\weight(\deg(y,G))}\;.
\end{equation*}
\subsubsection{Discrete-time Model}
This continuous-time model naturally contains another, discrete-time
model as follows. Define the stopping times
\begin{equation*}
S_n:=\inf \{t: \abs{\tree (t)} =n+1\},
\end{equation*}
then the Markov chain $\tree (S_n)$ is a randomly growing tree, where exactly one vertex is born at each time unit, and every newly born vertex chooses its parent at random, choosing $x$ with probability
\begin{equation*}
\frac{\weight(\deg (x,G))} {\sum_{y\in G}\weight(\deg(y,G))}\;.
\end{equation*}
if the $\tree (S_{n-1})=G$.

It was in this framework that Barab\'asi and Albert originally formulated their model \cite{barabasialbert}. The relation of the two models is discussed in detail in \cite{rtv}. As mentioned before, the questions we pose can be formulated equivalently in both models, but we will use the continuous-time version in our proofs, for reasons of convenience.

%% sub %%%%%%%%%%%%%%%%%%%%%%%%%%%%%%%%%%%%%%%%%%%%%%%%%%%%%%%%%%%%

\subsection{Some Additional Notation and Known Results}
\label{sec:AdditionalNotation}

Let $\birth _x$ be the birth time of vertex $x$,
\begin{equation}\label{eq:birthdef}
\birth _x:=\inf\{t>0\;:\,x\in\tree(t)\}\;.
\end{equation}
Let $\wait _x$ be the time we have to wait for the appearance of
vertex $x$, starting from the moment that its birth is actually
possible (e.g.~when no other vertex is obliged to be born before him).
Namely, let
\begin{enumerate}[(a)]
\item $\wait _{\emptyset}:=0$, 
\item $\wait _{y1}:=\birth _{y1}-\birth _y$, for any $y\in\vertices$,  
\item and $\wait _{yi}:=\birth _{yi}-\birth_{y(i-1)}$, for each 
$y\in\vertices$ and $i\ge 2$, $i\in \I$.
\end{enumerate}

Let the function $\widehat\varrho:(0,\infty)\to(0,\infty)$ be defined as
\begin{equation}
\label{eq:laprho}
\widehat\varrho(\lambda):=
\expect\sum_{j=1}^\childmax e^{-\lambda \birth_j}
=\sum_{j=1}^{\childmax} \prod_{i=0}^{j-1}
\frac{\weight(i)}{\lambda+\weight(i)}\;.
\end{equation}
The function $\widehat\varrho$ plays a central role in the theory of the branching processes related to our model, as discussed in \cite{rtv}.\footnote{The reason for the notation $\widehat\varrho$ is that this function is the Laplace transform of the density of the point process formed by birth times in the first generation of the tree.}
However, in the present work we use little of that relation -- instead, we list here the known results that we will use.
%The following results are already known.  (see \cite{rtv}):
\begin{enumerate}

\item The equation
 \begin{equation*}%\label{eq:rho=1}
 \widehat\varrho(\lambda)=1
 \end{equation*}
 has a unique root $\malt>0$. This $\malt$ is called the Malthusian parameter.

\item This $\malt$ gives the rate of exponential growth of the tree size almost surely. The normalized size of the tree converges almost surely to a random variable, which we denote by 
\begin{equation*}%\label{def:glob}
\glob := \lim_{t\to\infty}\E ^{-\malt t}|\tree(t)|\;.
\end{equation*}

\item $\glob$ is almost surely positive, and 
\begin{equation} \label{eq:globexpectation}
0 < \expect \glob < \infty, 
\end{equation}
which implies (also) that almost surely the process $\tree (t)$ does not blow up in finite time.

\item Moreover,
\begin{equation}\label{eq:var_glob_finite}
\expect \glob^2<\infty.
\end{equation}
\end{enumerate}

The first statement is in our setting obvious from the definition, since we have assumed $2\le\childmax<\infty$. The second and third are shown in \cite{rtv}. The last statement is also implicit from \cite{rtv} -- the variance is even calculated. Alternatively, the finiteness of the variance follows from Theorem 6.8.1 in \cite{jagers}, which states $L^2$ convergence of the normalized size under the condition $\expect[(\sum_{i=1}^\childmax e^{-\lambda \birth_i})^2]<\infty$, which is again obvious, since $\childmax<\infty$.
\begin{remark}
\label{def:alternativeConstruction}
The process $\tree (t)$ has an alternative construction, which we state 
here and refer to later. Define a countably infinite number of independent
random variables $\tilde{\wait}_x$, indexed with the elements of $\vertices$,
as follows.
Let $\tilde{\wait}_{\emptyset}=0$, and for 
$x=i_1i_2\hdots i_n$, let $\tilde{\wait}_x$ be exponentially distributed
with parameter $\weight(i_n-1)$. Denoting the parent of $x$ by 
$\parent{x}$, we define $\tilde{\birth}_{\emptyset}=0$ and 
\begin{equation*}
\tilde{\birth}_x=\tilde{\birth}_{\parent{x}}+\tilde{\wait}_{\parent{x}1}
+\tilde{\wait}_{\parent{x}2}+\hdots +\tilde{\wait}_{\parent{x}i_n}.
\end{equation*}
It is straightforward that with $\tilde{\tree}(t):=\{x\in \vertices:
\tilde{\birth}_x\le t\}$, the process $\tilde{\tree}$ has the same distribution as
$\tree$.
\end{remark}
%
%%%%%%%%%%%%%%%%%% sub %%%%%%%%%%%%%%%%%%%%%%%%%%%%%%%%%%%
\subsection{Limiting Objects}
\label{sec:LimitingObjects}
Let $\subtree{\tree}{x}(t)=\subtree{(\tree(t))}{x}$ denote the subtree of $\tree(t)$ rooted at $x$, which is the set of descendants of $x$ (including $x$) that are born up to time $t$. (Note that $t$ here is total time, and not the time since birth of $x$. In particular, $|\subtree{\tree}{x}(0)|=0$ if $x$ is not the root.)
For every $x\in \vertices$, we introduce the variables $\glob_x$,
corresponding to the growth of the subtree under $x$, analogously to $\glob$, 
\begin{equation*}
\glob _x := \lim _{t \to \infty} \E ^{-\malt (t-\birth _x)}
|\subtree{\tree}{x}(t)|\;.
\end{equation*}
The letter $\glob$ refers to the variable corresponding to
the root.
Clearly, for every $x\in\vertices$, the random variables $\glob _x$ are
identically distributed.
The basic relation between the different $\glob _x$
variables in the tree is that for any $x\in\vertices$,
\begin{equation}
\label{eq:decomposetheta}
\glob _x=\sum _{i=1}^{\childmax} 
\E ^{-\malt (\birth _{xi}-\birth _x)} \glob _{xi}\;,
\end{equation}
which is straightforward from 
$|\subtree{\tree}{x}(t)|=1+\sum_{i=1}^{\childmax}|\subtree{\tree}{xi}(t)|.$

Now let us ask the following question. Fix a vertex $x\in\vertices$, and at
time $t$, draw a vertex $\ver _t$ uniformly randomly from $\tree (t)$. 
What is the probability that $\ver _t$ is descendant of $x$, so 
$x \isancestor \ver _t$? As shown in (\ref{eq:delta-def}) below, this probability 
tends to an almost sure limit $\relat _x$ as $t\to\infty$, 
which can be expressed using the $\birth$ and $\glob$ random 
variables,
\begin{equation}
\label{eq:delta-def}
\relat _x := 
\lim _{t\to \infty} \frac{|\subtree{\tree}{x}(t)|}{|\tree (t)|} =
\E ^{-\malt \birth _x}\lim _{t\to \infty} 
\frac{ \E ^{-\malt (t-\birth _x)}
  |\subtree{\tree}{x}(t)|}
{\E ^{-\malt t}|\tree (t)|}=
\frac{\E ^{-\malt \birth _x} \glob _x}{\glob _{\emptyset}}\;.
\end{equation}
We can now, for any $n\in \N$, define a random measure $\measure_n$ on the 
finite set $\{x\;:\,|x|=n\}$ 
(on the $n^{\text{th}}$ generation of the tree),
by 
\begin{equation*}
\measure_n(\{x\}):=\relat _x\;. 
\end{equation*}
This is a probability measure
almost surely, which follows from the facts $\relat _{\emptyset}=1$ 
and $\relat _y=\sum_{i=1}^{\childmax}\relat _{yi}$. 

Let $H_n$ denote the entropy of $\measure_n$, that is
\begin{equation*}
%\label{def:H_n}
H_n=-\sum_{|x|=n}\relat _x \log \relat _x\;.
\end{equation*}

%%%%%%%%%%%%%%%%%% subsub %%%%%%%%%%%%%%%%%%%%%%%%%%%%%%%%%%%%%

\subsubsection{A Measure as the Limiting Object for the Tree}
\label{sec:measureAsLimiting}

Let $\leaves$ denote the set of leaves of the complete tree:
$\leaves=\{1,2,\dots ,\childmax\}^\infty$.
The concatenation $xy$ makes sense for $x\in \vertices$ and $y\in\leaves$,
and then $xy\in\leaves$. Also, for $x\in \vertices$ and $z\in\leaves$,
we write $x\isancestor z$ if $\exists y\in\leaves$ such that $z=xy$.
For $x\in\vertices$ we denote the set of leaves under $x$ by
$\cylinder{\leaves}{x}=\{z\in\leaves\;:\,x\isancestor z\}$.

Let $\leaves$ be equipped with the usual metric
\begin{equation}\label{eq:metric}
\metric (x,y)=\contr^{\max\{n\in \N\;:\,\crop{x}{n}=\crop{y}{n}\}}\;,
\end{equation}
where $0<\contr<1$ is an arbitrary constant. This constant is often chosen to be $1/e$,
which makes certain formulae appear simpler. Yet we will not fix the value, so that our
formulae express the dependence of the studied quantities on this arbitrary choice.

With the help of the $\measure _n$ random limiting measures, we define
$\measure$ on the cylinder sets $\cylinder{\leaves}{x}$ of $\leaves$ by
\begin{equation*}
\measure(\cylinder{\leaves}{x}):=\measure _n(\{x\})=\relat _x\;,\text{ if }\abs{x}=n\;,
\end{equation*}
and then we extend $\measure$ from $\{\cylinder{\leaves}{x}\,:\,x\in\vertices\}$ to the
sigma-algebra generated (on $\leaves$).
Our results concern the properties of this extended random measure $\measure$.

\begin{remark}
Now we can tell why we use the continuous and not the discrete-time model in our work. The limiting relative weights $\relat_x$ defined in (\ref{eq:delta-def}) also make sense and are interesting in the discrete-time setting, just like the measure $\measure$ and the entropy $H_n$. Our results are formulated in terms of these quantities. However, the limiting ``absolute'' weights $\glob_x$, which will play a central role in the proofs, don't make sense in the discrete-time setting.
\end{remark}

%%%%%%%%%%%%%%%%%%%%%%%%%%%%%%%%%%%%%%%%%%%%%%%%%%%%%%%%%%

\subsubsection{Dimensions of Measures: Definitions}

For the reader's convenience, let us review the definitions of local dimension, Hausdorff dimension and packing dimension of measures.
The lower and upper local dimensions of $\measure$ at $x$ are defined in \cite{Falconer97} (2.15) and (2.16) as
\begin{eqnarray}
 \underline{\dim}_\mathrm{loc} \measure(x) &=& \liminf_{r\to 0} \frac{\log\measure(B(x,r))}{\log r},\label{eq:lowerloc}\\
 \overline{\dim}_\mathrm{loc} \measure(x) &=& \limsup_{r\to 0} \frac{\log\measure(B(x,r))}{\log r},\label{eq:upperloc}
\end{eqnarray}
where $B(x,r)$ is the ball of radius $r$ centred at $x$. If the lower and upper local dimensions coincide at some $x$, they are called the local dimension at $x$.
The Hausdorff and packing dimensions of $\measure$ are defined  in \cite{Falconer97} (10.8) and (10.9) as
\begin{eqnarray}
 \dim_\mathrm{H} \measure &=& \sup\{s: \underline{\dim}_\mathrm{loc} \measure(x) \ge s \text{ for $\measure$-almost all $x$}\} ,\label{eq:dimH}\\
 \dim_\mathrm{P} \measure &=& \sup\{s: \overline{\dim}_\mathrm{loc} \measure(x) \ge s \text{ for $\measure$-almost all $x$}\}.\label{eq:dimP}
\end{eqnarray}
The name of these dimensions come from the fact (\cite{Falconer97} (10.10) and (10.11)) that
\begin{eqnarray}
 \dim_\mathrm{H} \measure &=& \inf\{\dim_\mathrm{H} E: \text{ $E$ is a Borel set with $\measure(E)>0$}\} ,\nn \\
 \dim_\mathrm{P} \measure &=& \inf\{\dim_\mathrm{P} E: \text{ $E$ is a Borel set with $\measure(E)>0$}\}.\nn
\end{eqnarray}
We are ready to state our results.

\subsection{Results}

\begin{theorem}
\label{thm:entropy}
The limiting entropy 
\begin{equation*}
h:=\lim_{n\to\infty} \frac{1}{n}H_n
\end{equation*}
exists and is constant with probability one. 
%(For $H_n$ recall (\ref{def:H_n}).)
\end{theorem}

\begin{theorem}
\label{thm:dimension}
The Hausdorff dimension $\dim_H \measure$ and the packing dimension $\dim_P \measure$ of the measure $\measure$ are 
constant and equal with probability one, and $h$ and the dimensions satisfy the relation
\begin{equation}
\label{eq:dimensionAndEntropy}
\dim_H \measure =\dim_P \measure = \frac{h}{-\log \contr},
\end{equation}
where $\contr$ is from (\ref{eq:metric}). Moreover, the local dimension of $\measure$ equals $\dim_H \measure =\dim_P \measure$ at
$\measure$-almost every point.
\end{theorem}

\begin{theorem}
\label{thm:explicit}
Furthermore, an explicit formula for $h$ is given:
\begin{equation*}
h = \expect \left(\sum_{i=1}^{\childmax} \malt \birth _i
e^{-\malt \birth _i} \right).
\end{equation*}
This can be computed given the weight function $\weight$. 
\end{theorem}

\subsection{Some Related Models and Results}
\label{sec:related_models}
In the last decades there has been much progress in describing the asymptotic structure of randomly
evolving trees, especially tree growth processes based on fragmentation processes. These processes are
closely related to our model, see Remark~\ref{rem:not_fragmentation}. Limiting objects called
``random real trees'' and ``continuum random trees'' were introduced, to which the evolving trees converge,
after an appropriate rescaling of the \emph{distances} on the tree. Much of the structure of these limiting
objects is understood, see e.g. \cite{HaasMiermont04,HaasMiermontPitmanWinkel08,HaasMiermont10}.

Our concept of the limiting measure $\measure$ is different from these. It is a measure on the set of leaves
of the infinite complete tree (with each vertex having exactly $K$ children), which is a metric space,
but the metric structure is trivial: it is not a result of any spatial scaling, and it carries no information about
the tree growth process. On the other hand, the weights given by $\measure$ are a result of an appropriate rescaling
of the tree size, where size means \emph{cardinality}. In short, we are really interested in the asymptotic weight
distribution, and not the asymptotic metric structure.
This asymptotic weight distribution is also studied in the Physics literature, see e.g. \cite{berestycki03},
where a quantity analogous to the local dimension is calculated for a continuous time fragmentation process.

Population growth models, studied excessively in the theory of branching processes (see e.g. \cite{jagers}), are
also intimately related to our model, as discussed in detail in \cite{rtv}.
Scientists discussed the Hausdorff dimension of the \emph{set} of individuals
that are actually (sooner or later) born. However, in our model this is uninteresting, because
-- almost surely -- every vertex is eventually born. Indeed, it is not the set, but the measure
which captures the long-term structure of the tree well, and of which the dimension is interesting.

Similarly, in the limiting continuous trees obtained in
\cite{HaasMiermont04,HaasMiermontPitmanWinkel08,HaasMiermont10} by a spatial rescaling of the evolving tree,
the metric structure is of main interest, and the Hausdorff dimension and Hausdorff measure of \emph{sets}
are the natural questions to ask \cite{DuquesneLeGall05,duquesne10} -- unlike in our setting.

The continuous time version of our tree growth process can also be translated into a branching random walk,
with time turning into displacement. Then the asymptotic growth can be described analogously,
see the Biggins theorem in \cite{Biggins77} or \cite{Lyons97}. 
However, with that point of view, the natural questions about the limiting structure are quite different.

% sub %%%%%%%%%%%%%%%%%%%%%%%%%%%%%%%%%%%%%%%%%%%%%%%%%%%%%%%%%%%%%%

\section{Main Line of the Proof}
\label{sec:mainline}

\subsection{Idea of the Proof}

The random limiting measure $\measure$ depends on the random growth of the tree. The idea of the proof is the following: we define a random leaf in the limiting tree according to the measure $\measure$. The way the random leaf is defined is based on a step-by-step construction of the subsequent generations of the limiting tree, together with a step-by-step construction of a path from the root to the random leaf. This is done in such a way that a Markov process appears naturally along this path, and the local dimension of the measure $\measure$ in this random point can be computed as an ergodic average. It follows that this average is constant with probability one, unconditionally. Thus, although the measure depends on the random tree growth, this ergodic average is constant, and it is the local dimension of the measure in all the $\measure$-typical leafs of the limiting tree. This implies that this constant is the Hausdorff (and also the packing) dimension of $\measure$ with probability one. Some technical difficulty comes from the fact that the state space of the key Markov process is continuous and non-compact, so to apply ergodic theorems, one has to work for the existence of the invariant measure (while uniqueness is easy).

\subsection{Markov Structure of the Tree}
\label{sec:Structure}

The content of this short section is mainly repetition of material from \cite{RudasToth2008}. These concepts and statements allow for a good understanding of the tree structure, on which our main construction (in Section~\ref{sec:ConstructionOfLeaf}) relies. Lemma~\ref{StrongMarkovLemma} will also be used formally in Section~\ref{sec:ConstructionOfLeaf} to get an easy proof of the fact that our step-by-step construction of the limiting tree is equivalent to the original model (Proposition~\ref{prop:constrequiv}).

\begin{definition}
We say that a system of random variables $(Y_x)_{x\in\vertices}$
constitutes a \emph{tree-indexed Markov field} if for any
$x\in\vertices$, the distribution of the collection of variables 
$\left(Y _y:\,x\isancestor y\right)$, and that of 
$\left(Y _z:\,x\not\isancestoreq z\right)$, are conditionally
independent, given $Y _x$.

\end{definition}

We state the following:

\begin{lemma}
\label{StrongMarkovLemma}

For each $x\in\vertices$ let $\waitglob _x$ denote the vector
$\waitglob _x:= (\wait_x,\glob _x)$. Then the collections of variables 
$\mathcal{A}_x:=\left(\waitglob _y:\,x\isancestor y\right)$ and 
$\mathcal{B}_x:=\left(\waitglob _z:\,x\not\isancestoreq z;\,
\wait _x\right)$ are conditionally independent, given $\glob _x$.

\end{lemma}

\begin{proof}
\smartqed

Recall Remark \ref{def:alternativeConstruction}, 
the alternative construction of $\tree (t)$. From that, it is 
straightforward that the collection
$\mathcal{A}_x$ is in fact constructed by the set of independent variables 
$A_x:=\left(\wait_y:\,x\isancestor y\right)$.

Similarly, recall (\ref{eq:decomposetheta}), 
and decompose $\glob _{\parent{x}}$, where $\parent{x}$ is
the parent of vertex $x$,
\begin{equation*}
%\label{eq:thetax}
\glob _{\parent{x}} = 
\sum _{j=1}^{\childmax} 
\E ^{-\malt (\birth _{\parent{x}j}-\birth _{\parent{x}})} 
\glob _{\parent{x}j} = 
\sum _{j=1}^{\childmax} 
\E ^{-\malt (\wait _{\parent{x}1}+
\wait_{\parent{x}2}+ \cdots + 
\wait _{\parent{x}j})} \glob _{\parent{x}j}\;.
\end{equation*}
This means that if we take the set of variables 
$B_x:=\left(\wait_y:\,x\not\isancestor y\right)$, then 
$\mathcal{B}_x$ is constructed by $B_x\cup \{\glob_x\}$.

Given $\glob_x$, the two collections
$A_x\cup\{\glob_x\}$ and $B_x\cup \{\glob_x\}$
are conditionally independent, this way the same is true for
$\mathcal{A}_x$ and $\mathcal{B}_x$, so the statement of the lemma
follows.

\end{proof}

\begin{corollary}
\label{GlobMarkovCorollary}
The variables $(\glob _x)_{x\in\vertices}$ constitute a tree-indexed
Markov field.
\end{corollary}

\begin{proof}
\smartqed
Direct consequence of Lemma \ref{StrongMarkovLemma}, since 
$\waitglob _x=(\wait _x,\glob _x)$.
\end{proof}

\begin{definition}
We introduce the variables $\fracrelat_x$, indexed by $\vertices$. For the root 
we leave $\fracrelat_{\emptyset}$ undefined. For any other vertex $y'$ which
has a parent $y$, so for any $y'=yi$ with $i\in\I$, let
\begin{equation*}
%\label{def:fracrelat}
\fracrelat_{yi}:=
\lim_{t\to\infty}
\frac{|\subtree{\tree}{yi}(t)|}{|\subtree{\tree}{y}(t)|}=
\frac{\E ^{   -\malt (\birth_{yi}-\birth_y)   } \glob_{yi}}{\glob_y}=
\frac{\relat_{yi}}{\relat_y}\;.
\end{equation*}
\end{definition}

Notice that for $x=(i_1i_2\hdots i_n)$, $\relat _x$ is a
telescopic product,
\begin{equation*}
\relat _x=\relat _{i_1}\frac{\relat_{i_1i_2}}{\relat_{i_1}}
\frac{\relat _{i_1i_2i_3}}{\relat _{i_1i_2}}
\hdots
\frac{\relat_ {i_1\hdots i_n}}{\relat _{i_1\hdots i_{n-1}}}=
\fracrelat_{i_1}\fracrelat_{i_1i_2}\fracrelat_{i_1i_2i_3}
\hdots\fracrelat_{i_1\hdots i_n}\;.
\end{equation*}
Equivalently, for $|x|=n$,
\begin{equation}
\label{eq:decomposelogrelat}
\log\relat _x=\sum_{l=1}^n \log \fracrelat_{x\mid _l},
\end{equation}
where $x\mid_l$ denotes the first $l$ letters of the string $x$ (which denotes the ancestor of $x$ on the $l$-th level of the tree).

% sub %%%%%%%%%%%%%%%%%%%%%%%%%%%%%%%%%%%%%%%%%%%%%%%%%%%%%%%%%%%%%%%%%
\subsection{Construction of the Random Leaf}
\label{sec:ConstructionOfLeaf}

We will now give a different construction of the tree from the ones seen before. Namely, we construct the system of $\waitglob _x=(\wait _x,\glob _x)$ variables starting from the root, and going step-by-step, from generation to generation. Together with these, we compute the $\fracrelat _x$ and $\relat _x$ variables, and use them to construct a random path $\{y_n\}$ from the root to the edge of the infinite tree. The $y_n$ will be chosen from the children of $y_{n-1}$ in a ``size-biased'' way. We will use this path in the proofs of our results. For the sake of simple notation, we suppose for a moment that the maximum number of children of any vertex is two, that is, $\childmax = 2$. It is straightforward to construct the corresponding generations and the random path for any $\childmax < \infty$. For the rest of this section we treat the distribution of $\glob$ as known.

Recall that 
$\wait _1$, $\wait _2$, $\glob _1$ and $\glob_2$ are independent. Keeping that in mind, using
\begin{equation}\label{eq:decomposethetaK=2}
\glob = \E ^{-\malt \wait _1}(\glob _1+\E^{-\malt \wait _2}\glob _2),
\end{equation}
we will consider the conditional joint distribution of $(\wait_2,\glob_1,\glob_2)$, given $\glob$. (Of course, $\wait_1$ is -- conditionally -- a deterministic function of these, but we will not use the value.)
%
% the conditional distribution of $\wait _1$, given $\glob$, is
% straightforward to calculate. Then, given $\glob$ and $\wait _1$, the
% conditional distribution of $\glob_1$ follows. After this, given now
% $\glob$, $\wait _1$ and $\glob _1$, the conditional distribution of
% $\wait _2$ can be determined. Finally, if we know $\glob$, $\wait _1$, $\glob
% _1$ and $\wait _2$, then $\glob _2$ is a deterministic function of
% these.
%
Now we can construct the generations, together with the random path $y_n$, in the following steps.
\begin{enumerate}
\item
Pick $\glob _{\emptyset}$ at random, according to its distribution, and fix $\wait _{\emptyset} = 0$. Also, fix $y_0={\emptyset}$.
\item \emph{First generation}
\begin{enumerate}
\item Pick $(\wait_2,\glob_1,\glob_2)$ according to their conditional
distribution, given $\glob _{\emptyset}$
\item Define $\relat _1=\fracrelat _1=\frac{\glob_1}{\glob_1+\E^{-\malt \wait _2}\glob_2}$
(which is equal to $\frac{\E^{-\malt \wait _1}\glob_1}{\glob}$, and happens not to depend on $\wait_1$).
Also define
$\relat _2=\fracrelat _2=\frac{\E^{-\malt \wait _2}\glob_2}{\glob_1+\E^{-\malt \wait _2}\glob_2}$.
\item
Choose $y_1$ according to the conditional probabilities $\prob (y_1 = 1|\glob,\wait_2,\glob_1,\glob_2) = \fracrelat _1$ and 
$\prob (y_1 = 2|\glob,\wait_2,\glob_1,\glob_2) = \fracrelat _2$.
% \item
% Pick $\waitglob _1 = (\wait _1, \glob _1)$ according to their conditional
% distribution, given $\glob _{\emptyset}$. These three numbers
% define $\relat _1=\fracrelat _1=\frac{\exp(-\malt \wait _1)\glob_1}{\glob}$.
% 
% \item
% Pick $\waitglob _2 = (\wait _2, \glob _2)$ similarly, according to their conditional
% distribution, given $\glob _{\emptyset}$ and $(\wait _1, \glob
% _1)$. At this point we can compute 
% $\relat _2=\fracrelat _2=\frac{\exp(-\malt (\wait _1+\wait
% _2))\glob_2}{\glob _{\emptyset}}$.
% 
% \item
% Choose $y_1$ according to $\prob (y_1 = 1) = \fracrelat _1$ and 
% $\prob (y_1 = 2) = \fracrelat _2$.
% 
\end{enumerate}
\item \emph{Second generation}
\begin{enumerate}
\item
Repeat the steps seen before for the progeny of vertex $1$, to get 
$(\wait_{12},\glob_{11},\glob_{12})$ and also
$\fracrelat _{11}$ and $\fracrelat _{12}$. This is done only using the information carried by $\glob _1$, conditionally independently of $(\glob,\glob_2)$. This conditional independence is the consequence of Corollary \ref{GlobMarkovCorollary}. Since we already know $\fracrelat _1$, we can now compute the values 
$\relat_{11}=\fracrelat_1\fracrelat _{11}$ and 
$\relat_{12}=\fracrelat_1\fracrelat _{12}$.
\item
Independently of the previous steps, use $\glob _2$ to get $(\wait_{22},\glob_{21},\glob_{22})$, $\fracrelat_{21}$ and $\fracrelat _{22}$. We then also have $\relat_{21}$ and $\relat_{22}$.
\item
Choose $y_2$ from the children of $y_1$, according to the conditional distribution given by the $\fracrelat _x$ variables in the second generation. Namely, if $y_1=1$, 
\begin{eqnarray*}
\prob (y_2=11 | y_1=1,\wait_{12},\glob_{11},\glob_{12})=\fracrelat_{11} \\
\prob (y_2=12 | y_1=1,\wait_{12},\glob_{11},\glob_{12})=\fracrelat_{12},
\end{eqnarray*}
and if  $y_1=2$,
\begin{eqnarray*}
\prob (y_2=21 | y_1=2,\wait_{22},\glob_{21},\glob_{22})=\fracrelat_{21} \\
\prob (y_2=22 | y_1=2,\wait_{22},\glob_{21},\glob_{22})=\fracrelat_{22}, 
\end{eqnarray*}
conditionally independently of the entire past of the construction.
\end{enumerate}
\item \emph{$n$-th generation} 
\begin{enumerate}
\item
Having constructed all the $\glob_x$ with $|x|=n-1$, split these all in the way above,
conditionally independently of each other (and the entire past of the construction), to get the $\fracrelat _z$ and $\relat _z$ variables in the $n-th$ generation. In particular,
$$\fracrelat_{xi}=\frac{\E^{-\malt(\wait_{x1}+\cdots +\wait_{xi})}\glob_{xi}}{\glob_x}.$$
\item
According to the value of $y_{n-1}$, choose $y_n$ from its children, according to the corresponding $\fracrelat _z$ distribution (conditionally independently of the entire past).
\end{enumerate}
\end{enumerate}

\begin{remark}\label{rem:branching_sizebiased}
As mentioned before, our model is intimately related to a branching process, as discussed in \cite{rtv}.
In branching processes, the idea of size biasing is not at all new,
as its importance is emphasized e.g. in \cite{LyonsPemantlePeres95}.
\end{remark}

\begin{remark}\label{rem:not_fragmentation}
This step-by-step construction of the tree is similar to the fragmentation processes discussed
e.g. in \cite{bertoin06}. There the usage of ``randomly tagged branches'' based on size-biased
choices is a standard technique, see \cite{bertoin06}, Section 1.2.3. Note however, that our
step-by-step construction is \emph{not} a fragmentation process in the classical sense. 
In particular, the sequence of measures $\measure_n$ is \emph{not} Markov: the process
also ``remembers'' the values $\glob_x$ which influence how the weight $\measure_n(\{x\})$
at $x$ is further ``fragmented''.
\end{remark}

\begin{proposition}
\label{prop:constrequiv}
With $\waitglob_x=(\wait_x,\glob_x)$ as before, the distribution of $\{\waitglob _x\}_{x\in\vertices}$ in the above
construction is identical to the distribution in the randomly
growing tree model.
\end{proposition}

\begin{proof}
The statement we are proving is about the joint distribution of countably infinitely many (real-valued) random variables, so this joint distribution can be viewed as a measure on $\R^\N$,
\footnote{we could write $([0,\infty)\times[0,\infty))^\vertices$, but a measure on this can be viewed as a special case of a measure on $\R^\N$.} with the $\sigma$-algebra of measurable sets being the $\sigma$-algebra generated by cylinder sets -- defined in terms of finitely many of the $\wait_x$ and $\glob_x$. So to prove that the two measures on $\R^\N$ -- given by the two constructions -- coincide, it is enough to see that they coincide on such cylinder sets.

In terms of joint distributions: It is enough to see that the distributions of $\{\waitglob _x\}_{x\in\vertices}$ coming from the two constructions have identical finite-dimensional marginals. In particular, it is enough to show that for every $n$, the distribution of $\{\waitglob _x\}_{x\in\vertices, |x|\le n}$ in the above construction is identical to the distribution in the randomly growing tree model.

This is easy to see by induction:
\begin{itemize}
 \item For $n=0$ we have chosen the law of $\glob_\emptyset$ properly by construction, also $\wait_\emptyset=0$ as it should be.
 \item For $n=1$, the $\{\waitglob _x\}_{x\in\vertices, |x|=1}$ are constructed to have the right conditional joint distribution, given $\glob_\emptyset$, so the $n=0$ statement implies the $n=1$ statement. In particular, the $\glob_x$ for $|x|=1$ are distributed as they should be.
 \item For $n\ge 2$, the same argument (the construction) gives inductively that the joint distribution of the $\{\waitglob_x\}_{x\in W}$ is what it should be, for any family $W$ of $x$-es which consists of a vertex and its children. However, the construction also ensures the conditional independence of $\{\waitglob_y\}_{x\isancestor y}$ and $\{\waitglob_z\}_{x\not\isancestoreq z}$ given $\glob _x$, as in Lemma~\ref{StrongMarkovLemma}. This, together with the joint distributions of the $\{\waitglob_x\}_{x\in W}$ (with $W$ as above) already characterizes the joint distribution of $\{\waitglob _x\}_{x\in\vertices, |x|\le n}$.
\end{itemize}
\end{proof}

From now on, we will use the alternative construction of the tree in our discussion, so Proposition~\ref{prop:constrequiv} is used all the time in the proof, but this will not be formally mentioned.

\begin{definition}
Denote by $\tree$ the $\sigma$-algebra generated by $\{\wait_x \mid x\in\vertices\}$, which contains the full tree evolution.
\end{definition}
Note that for any $x\in\vertices$, $\glob _x$ is measurable with respect to $\tree$, so $\tree$ is also the $\sigma$-algebra generated by $\{\wait_x, \glob_x \mid x\in\vertices\}$, namely all the data about the tree -- but not about the random leaf -- during the parallel construction of the tree and the random leaf just presented.

The usefulness of the random leaf we constructed is shown by the following:

\begin{lemma}~\label{lem:randomleaf_measure}
Conditioned on $\tree$,
the conditional distribution of the leaf $\,\lim_n y_n$ is exactly the measure
$\measure$. Similarly, the conditional distribution of $y_n$ is exactly
$\measure _n$.
\end{lemma}

\begin{proof}
 The second statement can be seen by induction: $\measure_0$ obviously gives weight $1$ to the single point $\emptyset=y_0$. Later, by construction of $y_{n+1}$, for any $x\in\vertices$ with $|x|=n$ and any $i\in\I$ we have $\prob(y_{n+1}=xi\,|\,y_n=x,\tree)=\fracrelat_{xi}$, so if we assume inductively that $\prob(y_n=x\,|\,\tree)=\measure_n(\{x\})=\relat_x$, then $\prob(y_{n+1}=xi\,|\,\tree)=\relat_x \fracrelat_{xi}=\relat_{xi}=\measure_{n+1}(\{xi\})$ for any $|xi|=n+1$, so $y_{n+1}$ is indeed distributed according to $\measure_{n+1}$.

The first statement is an immediate consequence of the second, since for any cylinder set $\cylinder{\leaves}{x}$, if $|x|=n$, we have $\prob(y_\infty \in \cylinder{\leaves}{x}\,|\,\tree)=\prob(y_n=x\,|\,\tree)=
\measure_n(\{x\})=\measure(\cylinder{\leaves}{x})$.
\end{proof}

\begin{corollary}
\label{cor:entropy_with_Markov}
Conditioned on the tree, the conditional expectation of
$-\log \relat _{y_n}$ is exactly $H_n$.
\end{corollary}

\begin{proof}
Indeed, by the above lemma,
$$\expect(-\log \relat _{y_n}\,|\,\tree)=-\sum_{|x|=n} \prob(y_n=x\,|\,\tree) \log \relat _x =-\sum_{|x|=n} \measure_n(\{x\}) \log \relat _x =-\sum_{|x|=n} \relat_x \log \relat _x=H_n.$$
\end{proof}

\subsection{Markov Processes Along the Random Path}
%%%%%%%%%%%%Innen jon a tetel bizonyitasa%%%%%%%%%%%%%%%

The key to the proof is the following easy observation:

\begin{proposition}\label{prop:Xmarkov}
The stochastic process $X_n=\glob _{y_n}$ ($n=0,1,2,\dots$) is a homogeneous Markov process. By ``homogeneous'' we mean that the transition kernel does not depend on $n$.
\end{proposition}

\begin{proof}
This is clear from the construction in Section~\ref{sec:ConstructionOfLeaf}. Indeed, when constructing $\glob _{y_n}$, only the value of $\glob _{y_{n-1}}$ is used, and the construction is the same on every level.
\end{proof}

The reason to construct in Section~\ref{sec:ConstructionOfLeaf} the entire tree of pairs $(\glob_x,\relat_x)$ step by step -- and not just the random path $\{y_n\}$ on an already existing tree -- was exactly to make the Markov property of $\glob_{y_n}$ obvious. A direct proof without the step-by-step construction would also not be hard, but according to our taste, the underlying phenomena are more transparent this way.

Based on this proposition and equation~(\ref{eq:decomposelogrelat}), the proof of our main results will be a reference to an appropriate ergodic theorem. However, there are two issues to deal with before. First, the state space of our Markov processes is continuous and even non-compact, so the unique existence of the invariant measure needs to be discussed. This is done in the next proposition. Second, the quantity $-\log \fracrelat_{y_n}$, of which we want to calculate the ergodic average, is not an observable on the state space of $X_n$, so this state space needs to be extended. This obvious extension will be done in Corollary~\ref{cor:Yirredmarkov}.

Before starting the main arguments, let us formulate, as a lemma, an easy observation about the distribution of $\glob$. We will use this in the arguments both for the uniqueness and the existence of the invariant measure of $X_n$.
From now on, we will use the notation $\R^+$ for the set of \emph{positive} real numbers: $$\R^+=(0,\infty).$$ It is important that $0$ is not included, e.g. when we speak of functions being continuous or nonzero on $\R^+$.

\begin{lemma}\label{lem:golb_density_nice}
 $\glob$ is absolutely continuous w.r.t. Lebesgue measure on $\R^+$, with a density function $\globdensity$ which is continuous and strictly positive on $\R^+$.
\end{lemma}

\begin{proof}
Start from the decomposition (\ref{eq:decomposetheta}). It shows that $\glob$ is of the form $\glob=\E ^{-\malt \wait _1} \widehat\glob$ where $\wait_1$ is independent of $\widehat\glob$, which immediately implies that $\glob$ must be equivalent to Lebesgue measure on the interval from zero to its maximal value. On the other hand, $\glob\ge e^{-\malt \wait_1} \glob_1 + e^{-\malt (\wait_1+\wait_2)} \glob_2$ implies that $\glob$ is not bounded, since $\glob_1$ and $\glob_2$ are independent and distributed as $\glob$, and their prefactors can be arbitrarily close to $1$.
The same decomposition, applied once again, also implies that the density $\globdensity$ is even a continuous function (more precisely, can be chosen to be continuous), since $\glob$ being absolutely continuous w.r.t. Lebesgue measure implies that so is $\widehat\glob$ (since $\childmax<\infty$), the density of which is once again smoothened by $\glob=\E ^{-\malt \wait _1} \widehat\glob$.
\end{proof}

For the discussion of the invariant measures, let $\kernel$ denote the transition kernel of $X_n$ -- that is, $\kernel(t)$ is the conditional distribution of $X_{n+1}$ under the condition $X_n=t$ (for every $t\in\R^+$). We also use it as the operator acting on measures by $\eta\kernel=\int_{\R^+}\kernel(t)\D \eta(t)$.

\begin{proposition}\label{prop:Xirred}
 The transition kernel $\kernel$ of the Markov process $X_n=\glob _{y_n}$ has exactly one invariant measure.
\end{proposition}

\begin{proof}
Recall that the decomposition (\ref{eq:decomposetheta}) is the key relation between the $\glob_x$-es of the different generations, on which the construction of $X_n$ -- and thus every property of the transition kernel -- is based. 

The key observation is that $\kernel(t)$ is equivalent to Lebesgue measure (on $\R^+$, of course) for every $t\in\R^+$. This (and more) is explicitly stated and proven in Lemma~\ref{lem:kernel_function_continuity}. However, since we feel that this statement is really intuitive, let us give a rough reasoning here as well.

First, Lemma~\ref{lem:golb_density_nice} implies that the distribution of $\glob$ is equivalent to Lebesgue measure on $\R^+$. Recall now the construction in Section~\ref{sec:ConstructionOfLeaf}, the essence of which is that $\kernel(t)$ is the conditional distribution of $\glob'$ under the condition $\glob=t$, where $\glob'$ is a random choice from the set $\{\glob_1,\dots,\glob_\childmax\}$. Look again at the relation between $\glob$ and $\{\glob_1,\dots,\glob_\childmax\}$, which is the decomposition (\ref{eq:decomposetheta}), or the simplified form for $K=2$, which is (\ref{eq:decomposethetaK=2}). It shows that given any value of $t$, the condition $\glob=t$ doesn't rule out any of the possible values of a $\glob_i$ with $1\le i\le \childmax$. Also, the conditioning on $\glob=t$ doesn't spoil the absolute continuity of $\glob_i$, and the method of randomly choosing $\glob'$ from $\{\glob_1,\dots,\glob_\childmax\}$ also preserves absolute continuity. With this, the key observation is shown. Again, see Lemma~\ref{lem:kernel_function_continuity} for a detailed proof.

This observation about $\kernel(t)$ implies that for any measure $\eta$ on $\R^+$, the first iterate $\eta\kernel$ is already equivalent to Lebesgue measure. This in turn implies that any invariant measure $\eta=\eta\kernel$ is equivalent to Lebesgue measure, so any two invariant measures are equivalent.

Suppose now indirectly that there exist two different invariant probability measures. Then two different \emph{extremal} invariant probability measures also have to exist. But two different extremal invariant probability measures must be mutually singular, which contradicts the previous argument. Thus there is at most one invariant probability measure.

The existence follows from Lemma~\ref{lem:weak-continuity} and Lemma~\ref{lem:markovlimit}. Indeed, the limiting measure $\markovlimitmeasure$ of Lemma~\ref{lem:markovlimit} has to be invariant by Lemma~\ref{lem:weak-continuity}.
\end{proof}

\begin{lemma} \label{lem:markovlimit}
The sequence of random variables $X_n=\glob _{y_n}$ is weakly convergent to some measure $\markovlimitmeasure$ on $\R^+$.
% The limit is equivalent to Lebesgue measure on $\R^+$.
\end{lemma}
To keep our arguments easy to follow, we delay the proof to Section~\ref{sec:markovlimitproof}.

\begin{lemma} \label{lem:weak-continuity}
 $\kernel$ is continuous with respect to weak convergence of measures.
\end{lemma}

The proof is delayed to Section~\ref{sec:weak-continuity}.

%%%%%%%%%%%%%%%%%%%%%%%%%%%%%%%%%%%%%%%%%%
\begin{corollary}\label{cor:Yirredmarkov}
The stochastic process $Y_n=(\glob _{y_n}, \fracrelat _{y_n})$ ($n=1,2,\dots$) is a homogeneous Markov process, for which the transition kernel has exactly one invariant measure.
\end{corollary}

\begin{proof}
Notice that during the construction of the tree in Section~\ref{sec:ConstructionOfLeaf}, $\fracrelat_{y_n}$ is constructed by using only the value of $\glob _{y_{n-1}}$ (not even $\fracrelat_{y_{n-1}}$), in a time-homogeneous way. Thus $Y_n$ is really homogeneous Markov.  Let $\Ykernel$ denote the transition kernel. From the construction, $\tilde\eta\Ykernel$ depends only on the first marginal of $\tilde\eta$, and on this marginal it acts exactly like $\kernel$. So for any measure $\hat\markovlimitmeasure$ with first marginal $\markovlimitmeasure$, $\tilde\markovlimitmeasure:=\hat\markovlimitmeasure\Ykernel$ is invariant by the invariance of $\markovlimitmeasure$ under $\kernel$. The uniqueness is obvious from the uniqueness of $\markovlimitmeasure$.
\end{proof}

Now we are ready to apply an ergodic theorem on the sequence $-\log\fracrelat _{y_n}$ to get the central technical result, from which our first two theorems easily follow.

\begin{corollary}
\label{cor:h_is_constant}
The limit $h:=-\lim_{n\to\infty}\frac{1}{n}\log \relat _{y_n}$ exists
and is constant with probability one.
\end{corollary}

\begin{proof}
$-\log\fracrelat _{y_n}$ is an observable on the state space of $Y_n$, and $h$ is exactly the ergodic average of this observable by 
(\ref{eq:decomposelogrelat}). So it is guaranteed to be constant by the unique existence of the invariant measure and Theorem 1.1 in Chapter X of \cite{Doob53}. We give the details of the (standard) argument now.

Theorem 1.1 in Chapter X of \cite{Doob53} states that ``If $\{x_n,n\ge 0\}$ is a stationary Markov process, and if $z$ is an invariant random variable, then $z$ is measurable on the sample space of $x_0$''. To formally apply this theorem to our process, we first need to construct a stationary version of $Y_n$. Namely, let $\tilde Y_n$ be the Markov process with generator $\Ykernel$ started from $\tilde Y_0$ which is distributed according to the unique invariant measure $\tilde\nu$. For this process, the ergodic average of an observable, being an invariant random variable (see \cite{Doob53}, Chapter X for the definition), is by the above theorem measurable on the state space -- that is, constant with probability one, conditioned on the initial value (more precisely, for $\tilde\nu$-a.e.\,initial value). But in our case, this constant is indeed independent of the initial value -- actually, it is constant for \emph{every} initial value, since $\Ykernel$ brings any measure (e.g.\,a point measure concentrated on any point) into a measure equivalent with $\tilde\nu$. Now notice that the property that the ergodic average is the same constant with probability one, independently of the initial state, is a property of the transition kernel $\Ykernel$ only (and not of $\tilde Y_n$ as a stochastic process), so it also holds for the process $Y_n$.
\end{proof}

Remember that $\frac{1}{n}H_n$ is a conditional expectation of $-\frac{1}{n}\log \relat _{y_n}$ by Corollary~\ref{cor:entropy_with_Markov}. So since we have just shown the almost sure convergence of $-\frac{1}{n}\log \relat _{y_n}$, the almost sure convergence of $\frac{1}{n}H_n$ follows, if we have e.g. dominated convergence. This will be guaranteed by the following lemma.

\begin{lemma} 
\label{lem:domkonv}
Let $\tmpmeasure$ be any Borel probability measure on $\leaves$, with $\childmax<\infty$. Using
the notation in Section~\ref{sec:measureAsLimiting}, for every $x \in \leaves$ let
\begin{equation*}
f_n(x) := -\frac{1}{n} \log \tmpmeasure ( \cylinder {\leaves}{x|_n} ).
\end{equation*}
Then $\bar{f} := \sup _n f_n$ is integrable with respect to the measure $\tmpmeasure$.
\end{lemma}

The proof is delayed to Section~\ref{sec:dominated}. Now we are ready to prove the main results of the paper.

\begin{proof}[Proof of Theorem \ref{thm:entropy}]
For every $x\in\leaves$ let $f_n(x)=-\frac{1}{n} \log \measure_n(\{x|_n\})=-\frac{1}{n} \log \measure ( \cylinder {\leaves}{x|_n} )$.
By Lemma~\ref{lem:randomleaf_measure}, Corollary~\ref{cor:h_is_constant} states exactly that for almost every realization of the tree, $f_n(x)$ converges $\measure$-almost surely to $h$.

Now divide the statement of Corollary~\ref{cor:entropy_with_Markov} by $n$ to get
\begin{equation*}
\frac{1}{n}H_n = \expect \left(-\frac{1}{n} \log \relat _{y_n} | \tree \right)
= \int_{\{\bar{x}\in \vertices:|x|=n\}} -\frac{1}{n} \log ( \measure _n (\{\bar{x}\})) \D \measure _n (\bar{x}) =
\int_{\leaves} f_n (x) \D \measure (x).
\end{equation*}
We can now apply the dominated convergence theorem to finish the proof, since we can use
the supremum as an integrable dominating function, see Lemma~\ref{lem:domkonv}.
\end{proof}

\begin{proof}[Proof of Theorem \ref{thm:dimension}]
We first show the second statement of the theorem by showing that the local dimension of $\measure$ at the leaf $\lim_n y_n$ is exactly $\frac{h}{-\log\contr}$ where $h$ is from Corollary~\ref{cor:h_is_constant}.
Let $B(x,r)$ denote the $r$-neighbourhood of the point $x\in \leaves$ w.r.t. the metric (\ref{eq:metric}). For $r=\contr^n$, this neighbourhood is formed exactly by the descendants of $x|_n$, so $B(x,\contr^n)=\leaves(x|_n)$. The $\measure$-measure of this set is
\begin{equation*}
\measure(B(x,\contr^n))= \measure ( \leaves(x|_n) ) = \measure _n (\{x|_n\}) = \log \relat _{x|_n},
\end{equation*}
while the logarithm of the diameter of this set is $n\log\contr$.
Thus the local dimension of $\measure$ at the leaf $x$ is
\begin{equation*}
\dim_\mathrm{loc}\measure(x) = \lim _{n\to\infty} \frac{\measure(B(x,\contr^n))}{n\log\contr} =
\lim _{n\to\infty} \frac{-\frac{1}{n}\log \relat _{x|_n}}{-\log \contr} 
\end{equation*}
(if this limit exists), by the definition in (\ref{eq:lowerloc}) and (\ref{eq:upperloc}).

Applying that to $x=\lim_n y_n$, Lemma~\ref{lem:randomleaf_measure} and Corollary~\ref{cor:h_is_constant} say that this limit indeed exists and is equal to $\frac{h}{-\log\contr}$ for $\measure$-almost every $x$, which is what we wanted to show.

The first statement of the theorem in now an immediate consequence of the definitions of the Hausdorff and packing dimension of a measure in (\ref{eq:dimH}) and (\ref{eq:dimP}).

\end{proof}

%%%%%%%%%%%%%%%%%%%%%%%%%%%%%%%%%%%%%%%%%%%%%%%%%%%%%%%%
\section{Proofs of Auxiliary Lemmas}
\label{sec:proofs}

\subsection{The Lemma for Dominated Convergence of the Entropies}
\label{sec:dominated}

In this section we prove Lemma \ref{lem:domkonv}.

\begin{proof}[Proof of Lemma \ref{lem:domkonv}]

For arbitrary $M<\infty$, let us define the set
\begin{equation*}
F _M ^{(n)} := \{ x: f _n(x) \ge M \} =
\{ x: -\frac{1}{n} \log \tmpmeasure ( \cylinder {\leaves}{x|_n} ) \ge M \} =
\{ x: \tmpmeasure ( \cylinder {\leaves}{x|_n} ) \le e^{-nM} \}.
\end{equation*}

Since $f_n$ takes constant values on the $\childmax ^n$ cylinder sets, we have
\begin{equation}
\label{eq:smallerthanexp}
\tmpmeasure(F _M ^{(n)}) \le \childmax^n e ^{-nM} = \left( \childmax e^{-M} \right) ^n.
\end{equation}

Now we define
\begin{equation*}
F _M := \{ x: \bar{f}(x) > M \} =
\bigcup _n \{ x: f _n(x) > M \} \subseteq \bigcup _n F _M ^{(n)}.
\end{equation*}

By (\ref{eq:smallerthanexp}), for $M > \log (2\childmax)$,
\begin{equation*}
\tmpmeasure (F _M) \le \sum _{n=1} ^{\infty} \left( \childmax e^{-M} \right) ^n <
2\childmax e^{-M}.
\end{equation*}

Thus, since $\bar {f} \ge 0$,
\begin{equation*}
\int \bar{f} (x) \D \tmpmeasure (x) < 
\sum _{M=1} ^{\infty} M \tmpmeasure ( \{ x: M-1 \le \bar{f} (x) < M \} ) < \infty.
\end{equation*}

\end{proof}

\subsection{Limiting Distribution of $\glob_{y_n}$ Along the Random Path} 
\label{sec:markovlimitproof}

In this section we prove Lemma \ref{lem:markovlimit}.
We begin with three lemmas of elementary probability whose statements do not rely on the setting of the paper.

The first one is a trivial generalization of the ordinary weak law of large numbers.
We could call it ``Weak law of large numbers with arbitrary weights''.
For this purpose, we will consider a sequence of probability vectors
$\{\underline{p}^n\}_{n=1}^\infty$, where, again, each $\underline{p}^n$ is a
probability vector $\underline{p}^n=(p_1^n,p_2^n,\dots,p_{N_n}^n)$. We plan to calculate
weighted averages of independent random variables with weight vectors $\underline{p}^n$.
We expect such an average to be close to the expectation, if every term has a sufficiently
small weight. So we will say that the sequence $\{\underline{p}^n\}_{n=1}^\infty$ is proper if
$$\lim_{n\to\infty}\max\{p_j^n:1\le j\le N_n\}=0$$.

\begin{lemma} \label{lem:weaklaw}
Let $\nu_0$ be a probability distribution on $\R$ with finite expectation $m$.
Let $\{\underline{p}^n\}_{n=1}^\infty$ be a proper sequence of weight vectors, and
let $\nu_n$ be the distribution of $$\sum_{j=1}^{N_n}p_j^n Z_j$$ where
$Z_1,Z_2,\dots, Z_{N_n}$ are independent random variables with distribution $\nu_0$.
Then $$\nu_n\Rightarrow m.$$
\end{lemma}

Note that this is the usual weak law if $p_j^n=\frac{1}{n}$ ($j=1,\dots,n$).

\begin{proof}
 The proof is trivial following the standard proof of the weak law with characteristic functions.
\end{proof}

Now we turn to a lemma which could be called ``size-biased sampling with arbitrary extra weights''.
For this purpose, let $\underline{p}=(p_1,p_2,\dots,p_N)$ be a probability vector, and let
$Z_1,Z_2,\dots ,Z_N$ be random variables on $\R^+$ (meaning $\prob(Z_j>0)=1$). We will say that the random variable $V$ is
the size-biased random choice from $Z_1,Z_2,\dots, Z_N$ with extra weights $p_1,p_2,\dots,p_N$,
if it is constructed the following way:
\begin{enumerate}
 \item Generate a realization of $(Z_1,Z_2,\dots, Z_N)$, and call it $(z_1,z_2,\dots, z_N)$.
 \item Having that, choose a random integer $J$ from the index set $\{1,2,\dots,N\}$ with the weight
 $$\frac{p_j z_j}{\sum_{j=1}^N p_j z_j}$$ given to each $j$.
 \item Set $V=z_J$.
\end{enumerate}
Note that this is the usual size-biased random choice if all the $p_j$ are equal. Our lemma states that this size-biased random choice with extra weights behaves just like the ordinary one, provided that every weight is small.

To state the lemma, let $\nu_0$ be a probability distribution on $\R^+$ with finite expectation $m$.
We will say that the measure $\nu$ is the size-biased version of $\nu_0$, if it is absolutely continuous with respect to $\nu_0$, and the density is $\rho(t)=\frac{1}{m}t$. In other words, $\nu(A)=\frac{1}{m}\int_A t \D\nu_0(t)$.

\begin{lemma} \label{lem:sizebiased}
Let $\nu_0$ be a probability distribution on $\R^+$ with finite expectation $m$.
Let $\{\underline{p}^n\}_{n=1}^\infty$ be a proper sequence of weight vectors, and (for
each $n$) let $Z_1^n,Z_2^n,\dots, Z_{N_n}^n$ be independent random variables with distribution $\nu_0$.
Let $V_n$ be the random choice from $Z_1^n,Z_2^n,\dots, Z_{N_n}^n$ with extra weights $p_1^n,p_2^n,\dots,p_{N_n}^n$. Let $\nu$ be the size-biased version of $\nu_0$.
Then
$$V_n\Rightarrow \nu.$$
\end{lemma}

\begin{proof}
 Let $F$ denote the cumulative distribution function of $\nu$, that is, $F(t)=\nu([0,t])$.
 Let $F_n$ denote the cumulative distribution function of $V_n$. For some fixed $t$, we write it in the form
 \begin{equation} \label{eq:completeprob}
  F_n(t)=\expect(\prob(V_n\le t\mid \{Z_j^n\}_{j=1}^{N_n})).
 \end{equation}
 The conditional probability inside is just the
 weight of $j$-s with $Z_j\le t$, so
 $$\prob(V_n\le t\mid \{Z_j^n\}_{j=1}^{N_n})=\frac{\sum_{j=1}^{N_n}p_j^n Z_j^n \ind(Z_j^n\le t)}{\sum_{j=1}^{N_n}p_j^n Z_j^n}.$$
 According to Lemma~\ref{lem:weaklaw} the denominator converges weakly (and thus, in probability) to $\expect(Z_1^n)=m>0$ as $n\to\infty$. Similarly, the numerator converges in probability to $$\expect(Z_1^n \ind(Z_1^n\le t))=\int_{\R^+}\tilde{t}\ind(\tilde{t}\le t)\D\nu_0(\tilde{t})=m \nu([0,t]).$$
This implies that the quotient converges weakly to $\nu([0,t])=F(t)$. Since this quotient is a conditional probability, it is obviously bounded by $1$, so (\ref{eq:completeprob}) implies that $F_n(t)\to F(t)$.
\end{proof}

The following lemma is just a re-statement of the previous one. This is the form that we will use.
\begin{lemma} \label{lem:sizebiasedeps}
 Let $\nu_0$ be a probability distribution on $\R^+$ with finite expectation, and let $\nu$ be its size-biased version. Let $\probe$ be a bounded continuous function on $\R^+$. Then for every $\eps>0$ there exists a $\delta>0$ such that for any probability vector $(p_1,p_2,\dots,p_N)$ which satisfies that
$$\max\{p_j : 1\le j\le N\}\le \delta,$$ if $Z_1,Z_2,\dots, Z_N$ are independent with distribution $\nu_0$, then the size-biased random choice (called $V$) from $Z_1,Z_2,\dots, Z_N$ with extra weights $p_1,p_2,\dots,p_N$ satisfies
$$|\expect(\probe(V))-\int \probe(t) \D\nu(t)|<\eps.$$
\end{lemma}

Before proving Lemma~\ref{lem:markovlimit}, we need one more tiny statement about the structure of the growing tree.

\begin{lemma} \label{lem:vertexprobtozero}
 For any vertex $x\in\vertices$, let 
\begin{equation}
\label{eq:vertexweight}
\vertexweight_x=e^{-\malt \birth_x},
\end{equation}

and for every $x$ with $|x|=n$ let $$\vertexprob_x=\frac{\vertexweight_x}{\sum_{|y|=n} \vertexweight_y}.$$
Then the sequence $\vertexprob^{n,\text{max}}:=\max\{\vertexprob_x : |x|=n\}$ converges to zero in probability.
\end{lemma}

\begin{proof}
 We prove the stronger statement that $\vertexprob^{n,\text{max}}$ converges to zero with probability one. We use the form
 \begin{equation} \label{eq:pmaxasfraction}
  \vertexprob^{n,\text{max}}=\frac{\max\{\vertexweight_x : |x|=n\}}{\sum_{|y|=n} \vertexweight_y}.
 \end{equation}

We show that the numerator converges to zero with probability one, while the denominator converges to a positive limit with probability one.
\begin{enumerate}
 \item If the numerator does not converge to zero, then there is some $\eps>0$ and there are infinitely many vertices $x\in\vertices$ with $\vertexweight_x>\eps$. Then, for all these $x$ we have $\birth_x<\birth^*:=\frac{-\log\eps}{\malt}$, so infinitely many vertices are born within the finite time $\birth^*$. This is known to have probability zero -- see comment at (\ref{eq:globexpectation}).

 \item Iterating the decomposition of $\glob$, we get
 \begin{equation} \label{eq:globdecompiterate}
  \glob=\sum_{|x|=n}\vertexweight_x \glob_x.
 \end{equation}

Let $\birthtimes_n$ denote the $\sigma$-algebra generated by $\{\wait_x : x\in\vertices, |x|\le n\}$ -- that is, the complete history of the tree growth up to the $n$-th level. Similarly, let $\birthtimes$ denote the $\sigma$-algebra generated by $\{\wait_x : x\in\vertices\}$. Clearly $\birthtimes_n\subset\birthtimes_{n+1}$, $\birthtimes$ is generated by $\cup_n \birthtimes_n$, and $\glob$ is $\birthtimes$-measurable. So L\'evy's `upward' theorem ensures that $\expect(\glob\mid\birthtimes_n)\to \glob$ with probability one. However, if $|x|=n$, then $\glob_x$ is independent of $\birthtimes_n$, while $\vertexweight_x$ is $\birthtimes_n$-measurable, so (\ref{eq:globdecompiterate}) implies that
$$\expect(\glob\mid\birthtimes_n)=\sum_{|x|=n}\vertexweight_x \expect \glob_x=\expect \glob \sum_{|x|=n}\vertexweight_x,$$
so with probability one the denominator of (\ref{eq:pmaxasfraction}) converges to
$\frac{\glob}{\expect \glob}\ne 0$.
\end{enumerate}

\end{proof}

\def \iexp {\xi}                 %independent exponentially distributed random variables
\def \sumexp {\Xi}               %sum of \iexp
\def \bern {B}                   %Bernoulli r.w.
\def \sumbern {S}                %sum of \bern times some coefficients

Now we can complete the goal of this subsection:

\begin{proof}[Proof of Lemma \ref{lem:markovlimit}]

Actually we give the limit explicitly. Let $\markovlimitmeasure$ be the measure on $\R^+$ with density function $cx\globdensity(x)$, where $\globdensity(x)$ is the density of $\glob$, and $c=\frac{1}{\expect\glob}$ is a normalizing constant. We will show that
\begin{equation} \label{eq:markovlimit}
 X_n\Rightarrow \markovlimitmeasure.
\end{equation}
Let us look directly at $X_n=\glob _{y_n}$ for some fixed $n$. This can also be constructed in the following way:
 \begin{enumerate}
 \item Generate the birth times $\birth_x$ for all vertices $x$ with $|x|=n$ (that is, on the $n$-th level of the tree). This defines the values $\vertexweight_x=e^{-\malt\birth_x}$, $|x|=n$. For better transparency, let us normalize these values to get a probability distribution on the $n$-th level of the tree: $\vertexprob_x:=\frac{\vertexweight_x}{\sum_{|z|=n}\vertexweight_z}$ (for $|x|=n$).
 \item Also generate the random variables $\glob _x$ for $|x|=n$, which are independent of the $\vertexprob_x$.
 \item Now $y_n$ is chosen from the points $|x|=n$ according to the distribution $\measure_n$, so the weight given to some $x$ is
 $$\frac{\relat_x}{\sum_{|z|=n}\relat_z}=\frac{\vertexweight_x \glob_x}{\sum_{|z|=n}\vertexweight_z \glob_z}=\frac{\vertexprob_x \glob_x}{\sum_{|z|=n}\vertexprob_z \glob_z}.$$
 \end{enumerate}
 So, having the values $\vertexprob_x$ fixed, the value $X_n=\glob _{y_n}$ is the result of a size-biased sampling from the independent random variables $\glob_x$, $|x|=n$, with additional weights $\vertexprob_x$ -- just like in the context of Lemma~\ref{lem:sizebiased} and Lemma~\ref{lem:sizebiasedeps}.

 Now we can prove (\ref{eq:markovlimit}). Let $\probe$ be a fixed bounded continuous function on $\R^+$,  let $\bound$ be an upper bound of $|\probe|$, and let $\expprobe=\int_{\R^+}\probe\D\markovlimitmeasure$ (which satisfies $|\expprobe |\le \bound $). Let $\eps>0$ be arbitrary.

Choose $\delta>0$ according to Lemma~\ref{lem:sizebiasedeps} so that if all the $\vertexprob_x$ on some level $|x|=n$ are at most $\delta$, then
\begin{equation*} 
\big| \expect \left( \probe (X_n)\mid \{\vertexprob _x\} \right )
-\expprobe \big|<\eps.
\end{equation*}

 Lemma~\ref{lem:vertexprobtozero} implies that there exists an $n_0$ such that for all $n>n_0$, $$\prob(\max\{\vertexprob_x: |x|=n\}>\delta)<\frac{\eps}{2\bound}.$$ Let $\event_{n,\delta}$ denote the event that $\max\{\vertexprob_x: |x|=n\}\le\delta$. For $n>n_0$ we get
 \begin{eqnarray*}
\big|\expect\left(\probe(X_n)\right)-\expprobe\big| \le
\int\big|\expect\left(\probe(X_n)-\expprobe\mid \{\vertexprob_x\}\right)\big|\D\prob =  \\
= \int_{\event_{n,\delta}^c}\big|\expect\left(\probe(X_n)-\expprobe\mid \{\vertexprob_x\}\right)\big|\D\prob + 
\int_{\event_{n,\delta}}\big|\expect\left(\probe(X_n)-\expprobe\mid \{\vertexprob_x\}\right)\big|\D\prob \le \\
\le 2 \bound \prob\left(\event_{n,\delta}^c\right) + \int_{\event_{n,\delta}}\eps\D\prob \le\eps+\eps=2\eps.
\end{eqnarray*}

\end{proof}

%%%%%%%subsection%%%%%%%%%%%%%%%%%%%%%%
\subsection{Weak Continuity of the Transition Kernel} \label{sec:weak-continuity}

This section is devoted to the proof of Lemma~\ref{lem:weak-continuity}.

\begin{proof}[Proof of Lemma~\ref{lem:weak-continuity}]
 We first show in Lemma~\ref{lem:kernel_function_continuity} that the transition kernel $\kernel$ can be written as $(\eta\kernel)(B)=\int_{\R^+}\int_B k(t,s) \D s \D \eta(t)$ where the kernel function $k(t,s)$ is continuous in the first variable (actually it is continuous in both variables). Lemma~\ref{lem:weak-continuity-general} -- which is a pure probability statement -- says that such a kernel is continuous with respect to weak convergence of measures.
\end{proof}

In the lemma, we show a little more than what is needed for the above proof. In particular, we also show that the kernel function $k(t,s)$ is nowhere zero on $\R^+ \times \R^+$, because this is used in the proof of Proposition~\ref{prop:Xirred}.

\begin{lemma}\label{lem:kernel_function_continuity}
 The transition kernel $\kernel$ can be written as $(\eta\kernel)(B)=\int_{\R^+}\int_B k(t,s) \D s \D t$ where
the kernel function $k(t,s)$ is continuous in both variables (in its domain $(t,s)\in \R^+ \times \R^+$), and
strictly positive.
\end{lemma}

\begin{proof}
For the time of the proof, let $\glob$ and $\glob'$ denote two consecutive values of the process, say $\glob:=X_n=\glob_{y_n}$, $\glob'=X_{n+1}=\glob_{y_{n+1}}$.
So the kernel function $k(t,s)$ is just the conditional density of $\glob'$ (as a function of $s$), under the condition $\glob=t$. So $$k(t,s)=\frac{\ggpdensity(t,s)}{\globdensity(t)},$$ where $\ggpdensity(t,s)$ is the joint density of the pair $(\glob,\glob')$, and $\globdensity(t)$ is its first marginal -- that is, the density of $\glob$.

We know from Lemma~\ref{lem:golb_density_nice} that $\glob$ is indeed absolutely continuous w.r.t. Lebesgue measure, and the density $\globdensity$ is continuous and nonzero on $\R^+$. Knowing this, we now show that $\ggpdensity(t,s)$ is also continuous in both variables and nonzero on $\R^+ \times \R^+$, which completes the proof.

We restrict to the case $K=2$. The case of a general $K<\infty$ causes no additional difficulty other than messy notation.
Following the construction of the tree in Section~\ref{sec:ConstructionOfLeaf}, we start with $\wait_1,\wait_2,\glob_1,\glob_2$ independent, with $\wait_i$ being exponentially distributed with parameter $\weight(i-1)/\malt$ and $\glob_i$ being distributed as $\glob$ ($i=1,2$). We introduce the temporary notation $S_i=\E ^{-\malt \wait _i}$ and denote its density by $g_i$. Explicit calculation gives that
\begin{equation}\label{eq:Sdensity}
g_i(u)=\frac{\weight(i-1)}{\malt} u^{\frac{\weight(i-1)}{\malt}-1} \ind_{(0,1)}(u),
\end{equation}
of which we will only use that $u\,g_1(u)$ is bounded.

Denote the joint density of $(S_1,S_2,\glob_1,\glob_2)$ by
$$f(u_1,u_2,t_1,t_2)=g_1(u_1) g_2(u_2) \globdensity(t_1) \globdensity(t_2).$$
We define
$$\glob=S_1 \glob_1 + S_1 S_2 \glob_2 = S_1 (\glob_1 + S_2 \glob_2).$$
To get the appropriate joint distributions, in the probability vector $(S_1,S_2,\glob_1,\glob_2)$
we replace $S_1$ by $\glob$, so let us denote the joint density of $(\glob,S_2,\glob_1,\glob_2)$ by $\tilde{f}$.
The density transformation formula gives
$$\tilde{f}(t,u_2,t_1,t_2)=\frac{1}{t_1+u_2 t_2} f(\frac{t}{t_1+u_2 t_2},u_2,t_1,t_2)=
\frac{1}{t} \frac{t}{t_1+u_2 t_2} g_1(\frac{t}{t_1+u_2 t_2}) g_2(u_2) \globdensity(t_1) \globdensity(t_2).$$
According to the construction, $\glob'$ is chosen to be either $\glob_1$ or $\glob_2$, with conditional probabilities (given $(S_2,\glob_1,\glob_2)$ and conditionally independently of $\glob$)
$$\prob(\glob'=\glob_1|S_2,\glob_1,\glob_2)=\frac{\glob_1}{\glob_1+S_2 \glob_2},$$
$$\prob(\glob'=\glob_2|S_2,\glob_1,\glob_2)=\frac{S_2 \glob_2}{\glob_1+S_2 \glob_2}.$$
So the joint density of $(\glob,\glob')$ is
\begin{eqnarray} \label{eq:ggpdensity_integral}
\ggpdensity(t,s)&=&\iint_{\R^2}\frac{t_1}{t_1+u_2 t_2}\tilde{f}(t,u_2,s,t_2)\D t_2 \D u_2+
\iint_{\R^2}\frac{u_2 t_2}{t_1+u_2 t_2}\tilde{f}(t,u_2,t_1,s)\D t_1 \D u_2= \nn \\
&=&\iint_{\R^2}\bar{f}_1(t,s,u_2,t_2)\D t_2 \D u_2+
\iint_{\R^2}\bar{f}_2(t,s,u_2,t_1)\D t_1 \D u_2
\end{eqnarray}

All there is left is to show that both integrals on the right hand side are continuous and nonzero for $(t,s)\in \R^+ \times \R^+$.
Now the integrands $\bar{f}_1$ and $\bar{f}_2$ are not exactly continuous, but they are continuous \emph{on their supports}. 
\footnote{\label{ftn:ggpdensity_integrand_support}The supports of the two integrands are actually not the same. Both of them are characterized by the system of inequalities $\{0<t_1,t_2$; $0<u_2<1$; $0<\frac{t}{t_1+u_2 t_2}<1\}$, but with the choice $s=t_1$ or $s=t_2$, respectively.}
On the other hand, for every $(t,s)\in \R^+ \times \R^+$ the support of each integrand is a nice set (described in the footnote) with a boundary of Lebesgue measure zero. That is, for \emph{every} $(t_0,s_0)\in \R^+ \times \R^+$,
$$\bar{f}_1(t,s,u_2,t_2)\stackrel{(t,s)\to (t_0,s_0)}{\longrightarrow}\bar{f}_1(t_0,s_0,u_2,t_2)
\text{ for Lebesgue-a.e. $(u_2,t_2)\in \R^2$}.$$
To get the desired continuity of the first integral by the Lebesgue dominated convergence theorem, we only need to find an integrable (in $(u_2,t_2)$) uniform (in $(t,s)$ near $(t_0,s_0)$) upper bound for
$$\bar{f}_1(t,s,u_2,t_2)=\frac{s}{s+u_2 t_2} \frac{1}{t} \frac{t}{s+u_2 t_2} g_1(\frac{t}{s+u_2 t_2}) g_2(u_2) \globdensity(s) \globdensity(t_2).$$
The first factor is at most $1$, and the product
$\frac{t}{s+u_2 t_2} g_1(\frac{t}{s+u_2 t_2})$ is bounded because $u\,g_1(u)$ is bounded due to (\ref{eq:Sdensity}). So we have
$$\bar{f}_1(t,s,u_2,t_2)\le C \frac{1}{t} \globdensity(s) g_2(u_2) \globdensity(t_2)
\le C (\frac{1}{t_0}+1) (\globdensity(s)+1) g_2(u_2) \globdensity(t_2)$$
if $(t,s)$ is close enough to $(t_0,s_0)$, since $\frac{1}{t}\globdensity(s)$ is continuous in $(t_0,s_0)$.
This upper bound is clearly integrable in $(u_2,t_2)$, so the dominated convergence theorem ensures that the integral is also continuous.

The second integral in (\ref{eq:ggpdensity_integral}) can be shown to be continuous in exactly the same way. Thus the continuity of $k(t,s)$ is proven.

To get that $\ggpdensity(t,s)$ (and thus $k(t,s)$) is strictly positive on $\R^+ \times \R^+$, we only need to note that the support of the integrand is nonempty for every $(t,s)\in \R^+ \times \R^+$ in both integrals on the right hand side of (\ref{eq:ggpdensity_integral}). This comes again from (\ref{eq:decomposethetaK=2}), which shows that any pair of positive values is possible for $(\glob,\glob_1)$ (in case of the first integrand) or for $(\glob,\glob_2)$ (in case of the second integrand). (See the footnote \ref{ftn:ggpdensity_integrand_support} for explicit formulae.) The integrands are of course also non-negative, so both integrals are positive.
\end{proof}

\begin{lemma} \label{lem:weak-continuity-general}
 Let $k:\R^+\times\R^+ \to [0,\infty)$ be a function continuous in the first variable, such that for every $t\in\R^+$ the function $k(t,.)$ is a probability density on $\R^+$ -- that is, $\int_{\R^+} k(t,s) \D s = 1$. Let the operator $P$ be defined on Borel probability measures of $\R^+$ by $$(\eta P)(B):=\int_{\R^+}\int_B k(t,s) \D s \D \eta(t)$$ for every Borel probability measure $\eta$ on $\R^+$ and every Borel set $B\subset\R^+$. Then $P$ is continuous with respect to weak convergence of measures.
\end{lemma}

This lemma is an easy consequence of the following:

\begin{lemma} \label{lem:mixed-continuity}
 Let $k:\R^+\times\R^+ \to [0,\infty)$ be a function as in Lemma~\ref{lem:weak-continuity-general}, and for every $t\in\R^+$ let $K_t$ denote the measure on $\R^+$ with density $k(t,.)$. Then if $t_n$ is a sequence in $\R^+$ converging to $t$, then $K_{t_n}$ converges to $K_t$ weakly.
\end{lemma}

\begin{proof}
By assumption, $\{k(t_n,.)\}_{n=1}^\infty$ is a sequence of density functions converging pointwise to the density function $k(t,.)$. This implies weak convergence of the corresponding measures through the Fatou lemma: for any Borel set $B\subset\R^+$
$$\liminf_{n\to\infty} K_{t_n} (B) = \liminf_{n\to\infty} \int_B k(t_n,s)\D s
\stackrel{\text{Fatou}}{\ge} \int_B \liminf_{n\to\infty} k(t_n,s)\D s=\int_B k(t,s)\D s=K_t(B),$$
similarly
$$\liminf_{n\to\infty} K_{t_n} (B^c) \ge K_t(B^c),$$
which implies
$$\limsup_{n\to\infty} K_{t_n} (B) =1-\liminf_{n\to\infty} K_{t_n} (B^c) \le 1-K_t(B^c)=K_t(B).$$
These together give
$$K_{t_n}(B) \to K(B).$$
\end{proof}

\begin{proof}[Proof of Lemma \ref{lem:weak-continuity-general}]
Let $\phi:\R^+\to\R$ be bounded and continuous and let $\eta_n$ be a sequence of measures on $\R^+$ converging weakly to $\eta$. By the definition of $P$,
\begin{eqnarray*}
 \int_{\R^+} \phi \D(\eta_n P)&=&\int_{\R^+\times\R^+} k(t,s) \phi(s) \D(\eta_n(t)\times\Leb(s))=\\
 &=&\int_{\R^+} \left[ \int_{\R^+} k(t,s)\phi(s)\D s \right] \D\eta_n(t).
\end{eqnarray*}
The function
$$\bar\phi(t):=\int_{\R^+} k(t,s)\phi(s)\D s$$
is obviously bounded, and also continuous: this is exactly the statement of Lemma~\ref{lem:mixed-continuity}. But then the weak convergence of $\eta_n$ to $\eta$ means exactly that
$$\int_{\R^+} \bar\phi(t)\D\eta_n(t)\to \int_{\R^+} \bar\phi(t)\D\eta(t),$$
so we have
$$\int_{\R^+} \phi \D(\eta_n P)\to \int_{\R^+} \bar\phi(t)\D\eta(t)=\int_{\R^+} \phi \D(\eta P)$$
for every bounded continuous $\phi$, which is exactly what we want to prove.
\end{proof}

%%%%%%%%% Computation of the entropy %%%%%%%%%%%%%%%%%%%%%%%%%%%%%%%%%%%%%
\section{Computation of the Entropy}
\label{sec:ComputationOfEntropy}

\begin{proof}[Proof of Theorem \ref{thm:explicit}]

We know that $\frac{1}{n}H_n= -\frac{1}{n}\sum_{|x|=n}\relat _x \log
\relat _x$ converges almost surely to some constant $h$, and this
constant is equal to the limit of the expected values. For this
section we use the shorthand notation already introduced in (\ref{eq:vertexweight}),
\begin{equation}\label{eq:vertexweight_again}
\vertexweight_x=e^{-\malt \birth _x}.
\end{equation}

To compute $h$, first observe that

\begin{eqnarray*}
\expect \sum_{|x|=n}\relat _x \glob \log (\relat _x \glob)=
\expect \left(\sum_{|x|=n}\glob \relat _x \log \relat _x\right)+
\expect \left((\glob \log \glob)\sum_{|x|=n} \relat _x \right)= \\
\expect \left(\glob \sum_{|x|=n}\relat _x \log \relat _x\right)+
\expect \left( \glob \log \glob \right),
\end{eqnarray*}

where we have used that $\sum_{|x|=n}\relat _x=1$ by definition.

Next we observe that on the other hand, the same expression can be
written as

\begin{eqnarray*}
\expect \sum_{|x|=n}\relat _x \glob \log (\relat _x \glob)=
\expect \sum_{|x|=n}T_x\glob _x  
\log \left(T_x\glob _x\right)= \\
\expect \left(\sum_{|x|=n}\glob _x  T_x
\log \left(T_x\right)\right)+
\expect \left(\sum_{|x|=n}T_x\glob _x  \log \glob _x\right)= \\
\sum_{|x|=n}\left(\expect \glob_x\right)\expect\left(T_x
\log T_x\right)+
\sum_{|x|=n}\expect \left(T_x\right)
\expect\left(\glob _x \log \glob _x\right)= \\
\left(\expect \glob\right)
\expect \sum_{|x|=n}\left(T_x
\log T_x\right)+
\expect\left(\glob \log \glob\right)
\expect\left(\sum_{|x|=n}T_x\right),
\end{eqnarray*}
where we have used that for any $x\in \vertices$, 
$\glob_x$ and $\birth_x$ are independent.
Recall that $\expect\left(\sum_{|x|=n}T_x\right)=1$.

Since (\ref{eq:var_glob_finite}) implies that $\expect\left(\glob \log \glob\right)<\infty$, comparing the two formulae gives the conclusion
\begin{equation}
\label{eq:show}
\expect \left(\glob \sum_{|x|=n}\relat _x \log \relat _x\right)=
\left(\expect \glob\right) \expect \left(\sum_{|x|=n}T_x
\log T_x \right).
\end{equation}

We compute the right-hand side with an induction,
\begin{eqnarray*}
A_n:=\expect \left( \sum_{|x|=n}T_x\log T_x \right)=
\expect \left( \sum_{|y|=n-1}\sum_{i=1}^{\childmax} T_{yi}\log T_{yi}\right)= \\
\left(\expect \sum_{i=1}^{\childmax} e^{-\malt(\birth_{yi}-\birth_y)} \right)
\expect\left(\sum_{|y|=n-1}T_y\log T_y\right)+ \\
\left(\expect \sum_{|y|=n-1} T_y \right)
\expect\left(\sum_{i=1}^{\childmax} e^{-\malt(\birth_{yi}-\birth_y)}
\log e^{-\malt(\birth_{yi}-\birth_y)}\right)= \\
A_{n-1}+\expect\left(\sum_{i=1}^{\childmax} T_i \log T_i \right),
\end{eqnarray*}

so 

\begin{equation*}
A_n = n \expect\left(\sum_{i=1}^{\childmax} T_i \log T_i\right).
\end{equation*}

Now write this back to (\ref{eq:show}) to get

\begin{equation*}
\expect \left(\glob \frac{1}{n} H_n \right)=
\left(\expect \glob\right) 
\expect \left(- \sum_{i=1}^{\childmax} T_i \log T_i \right).
\end{equation*}

Since $\lim \frac{1}{n}H_n=h$ almost surely and $\expect \glob < \infty$, we can apply the dominated convergence theorem if we check that $\frac{1}{n}H_n$ is bounded. This follows from the standard upper bound for entropy of measures on the finite set $\{x\in\leaves:|x|=n\}$, which has $\childmax^n$ elements, coming from the Jensen inequality:
\begin{eqnarray*}
H_n=-\sum_{|x|=n}\measure_n(\{x\}) \log \measure_n(\{x\})=\int\limits_{\{x\in\leaves:|x|=n\}} \log \frac{1}{\measure_n(\{x\})}\D \measure_n(x) \stackrel{\text{Jensen}}{\le} \\
\le \log \int\limits_{\{x\in\leaves:|x|=n\}} \frac{1}{\measure_n(\{x\})}=\log \sum_{|x|=n} \measure_n(\{x\})\frac{1}{\measure_n(\{x\})}=\log \childmax^n=n\log\childmax,
\end{eqnarray*}

so $\frac{1}{n}H_n\le\log\childmax$. Now dominated convergence gives 

\begin{equation*}
h = \expect \left( - \sum_{i=1}^{\childmax} T_i \log T_i \right).
\end{equation*}

Recalling (\ref{eq:vertexweight_again}), the proof of the theorem is complete.\end{proof}

\begin{remark}
This value can be explicitly calculated, as soon as the weight function is given, since the $\birth_i$ variables are the sum of independent, exponentially distributed random variables with parameters 
$\left(\weight (j)\right)_{j=0}^{i-1}$. Alternatively, with the function $\widehat\varrho$ defined in (\ref{eq:laprho}),
$$h=\malt \frac{\D \widehat\varrho(\lambda)}{\D\lambda} 
\Big|_{\lambda=\malt}.$$
\end{remark}

\section{Outlook}
The present result is restricted to the $\childmax<\infty$ case, i.e. when a vertex can only have finitely many children. This property is used in three places. First, Theorem~\ref{thm:entropy} relies on Lemma~\ref{lem:domkonv}, which is a very rough estimate working for finite $\childmax$ only. Second, in the proof of Theorem~\ref{thm:explicit} we use the fact that $\frac{1}{n}H_n$ is bounded -- which is also certainly false for $\childmax=\infty$. Third, showing the continuity of the density $\globdensity$ and the transition kernel function $k$ (in lemmas \ref{lem:golb_density_nice} and \ref{lem:kernel_function_continuity}) is easier using the fact that the sum in (\ref{eq:decomposetheta}) is finite. With more care, these could possibly be generalized for the $\childmax=\infty$ case, so the main result about the Hausdorff dimension, Theorem~\ref{thm:dimension} could be shown in greater generality. However, not having the explicit formula of Theorem~\ref{thm:explicit} is a serious drawback. We believe that the problem can be solved -- and the validity of the explicit formula can be shown -- for a large class of rate functions with $\childmax=\infty$ by a detailed analysis of the transition kernel $\kernel$. Such an analysis could be avoided in the present paper by the study of the limiting distribution in Section~\ref{sec:markovlimitproof}. We plan to return to that in the future.

%%%%%%%%% Acknowledgements %%%%%%%%%%%%%%%%%%%%
\section*{Acknowledgements}
We gratefully thank Bal\'azs R\'ath for the simple proof of Lemma~\ref{lem:sizebiased}. We are also grateful to an anonymous referee for many useful suggestions that helped improve the quality of the paper. A. Rudas acknowledges the support of OTKA grant K60708. I. P. T\'oth acknowledges the support of OTKA grants PD73609 and K71693, and is also grateful to the European Research Council for support. 

%%%%%%%%%%%%%%%%%%%%%%%%%%%%%%%%%%%%%%%%%%%%%%%%%%%

\bibliography{tree_entropy_new}{}
\bibliographystyle{plain}

\end{document}